\documentclass[a4paper,12pt]{article}

\usepackage{amsfonts,amssymb,amsmath,amsthm,latexsym,epsfig,amscd,bbm,stmaryrd, mathrsfs}

\usepackage[english]{babel}
\usepackage{graphicx}
\usepackage{tikz}
\usepackage{courier}
\usepackage{bbm}
\usepackage{enumerate}
\usepackage{psfrag}
\usepackage{wasysym}
\usepackage{type1cm}
\usepackage{bookmark}

\usepackage{calc}

\def\arraypar#1{\parbox[c]{\textwidth - 2cm}{\centering #1}}
\usepackage{environ}
\NewEnviron{display}{
\begin{equation}\begin{array}{c}
\arraypar{\BODY}
\end{array}\end{equation}
}

\def\clap#1{\hbox to 0pt{\hss#1\hss}}

\makeatletter \@addtoreset{equation}{section}
\makeatother

\makeatletter \@addtoreset{enunciato}{section}
\makeatother

\newcounter{enunciato}[section]

\newtheorem{ittheorem}{Theorem}
\newtheorem{itlemma}{Lemma}
\newtheorem{itproposition}{Proposition}
\newtheorem{itdefinition}{Definition}
\newtheorem{itremark}{Remark}
\newtheorem{itclaim}{Claim}
\newtheorem{itfact}{Fact}
\newtheorem{itconjecture}{Conjecture}
\newtheorem{itcorollary}{Corollary}
\newtheorem{itexample}{Example}

\newenvironment{theorem}{\addtocounter{enunciato}{1}
\begin{ittheorem}}{\end{ittheorem}}

\newenvironment{lemma}{\addtocounter{enunciato}{1}
\begin{itlemma}}{\end{itlemma}}

\newenvironment{proposition}{\addtocounter{enunciato}{1}
\begin{itproposition}}{\end{itproposition}}

\newenvironment{definition}{\addtocounter{enunciato}{1}
\begin{itdefinition}}{\end{itdefinition}}

\newenvironment{remark}{\addtocounter{enunciato}{1}
\begin{itremark}}{\end{itremark}}

\newenvironment{conjecture}{\addtocounter{enunciato}{1}
\begin{itconjecture}}{\end{itconjecture}}

\newenvironment{corollary}{\addtocounter{enunciato}{1}
\begin{itcorollary}}{\end{itcorollary}}

\newenvironment{example}{\addtocounter{enunciato}{1}
\begin{itexample}}{\end{itexample}}

\newcommand{\be}[1]{\begin{equation}\label{#1}}
\newcommand{\ee}{\end{equation}}

\newcommand{\bl}[1]{\begin{lemma}\label{#1}}
\newcommand{\el}{\end{lemma}}

\newcommand{\br}[1]{\begin{remark}\label{#1}}
\newcommand{\er}{\end{remark}}

\newcommand{\bt}[1]{\begin{theorem}\label{#1}}
\newcommand{\et}{\end{theorem}}

\newcommand{\bd}[1]{\begin{definition}\label{#1}}
\newcommand{\ed}{\end{definition}}

\newcommand{\bp}[1]{\begin{proposition}\label{#1}}
\newcommand{\ep}{\end{proposition}}

\newcommand{\bc}[1]{\begin{corollary}\label{#1}}
\newcommand{\ec}{\end{corollary}}

\newcommand{\bcj}[1]{\begin{conjecture}\label{#1}}
\newcommand{\ecj}{\end{conjecture}}

\newcommand{\bpr}{\begin{proof}}
\newcommand{\epr}{\end{proof}}

\def\Z{\mathbb{Z}}
\def\N{\mathbb{N}}
\def\R{\mathbb{R}}

\def\P{\mathbb{P}}
\def\E{\mathbb{E}}

\def \ba {\begin{array}}
\def \ea {\end{array}}

\def \P  {{\mathbb P}}
\def \E  {{\mathbb E}}

\usepackage{graphicx}

\DeclareSymbolFont{symbolsC}{U}{pxsyc}{m}{n}
\DeclareMathSymbol{\opentimes}{\mathrel}{symbolsC}{93}

\newcounter{constant}

\usepackage{array}
\newcolumntype{e}{>{\displaystyle}r @{\,} >{\displaystyle}c @{\,} >{\displaystyle}l}
\begin{document}

%%%%%%%%%%%%%%%%% TITLE PAGE %%%%%%%%%%%%%%%%%%%%%%%%%
\title{Zero-one law for directional transience of one-dimensional random walks in dynamic random environments}
\author{Tal Orenshtein and Renato S.\ dos Santos}

\maketitle

\begin{abstract}
We prove the trichotomy between transience to the right, transience to the left and recurrence
of one-dimensional nearest-neighbour random walks in dynamic random environments under fairly general assumptions,
namely: stationarity under space-time translations, ergodicity under spatial translations,
and a mild ellipticity condition.
In particular, the result applies to general uniformly elliptic models
and also to a large class of non-uniformly elliptic cases that are i.i.d.\ in space and Markovian in time.
An immediate consequence is the recurrence of models that are symmetric with respect to reflection through the origin.

\vspace{0.5cm}
\noindent {\it MSC} 2010.
Primary 60F20, 60K37; Secondary 82B41, 82C44.
\end{abstract}

%%%%%%%%%%%%%%%%%%%%%%%%%%%%%%%%% INTRO %%%%%%%%%%%%%%%%%%%%%%%%%%%%%%%%%%%%%%%%%%%%%%%%%%%%%%%%%%%%%%%%%%%%%%%%%%%%%%%%%%
\section{Introduction}
\label{s:intro}

Random walks in random environments have been the subject of intensive mathematical study for several decades.
They consist of random walks whose transition kernels are themselves random, modelling the movement
of a tracer particle in a disordered medium. When the random transition kernels, called \emph{random environment},
 do not evolve with time, the model is called \emph{static}; otherwise it is called \emph{dynamic}.
Hereafter we will use the abbreviations RWRE for the static and RWDRE for the dynamic model.
The reader is referred to the monographs \cite{Sz06}, \cite{Ze04} for RWRE
and \cite{Avthesis}, \cite{dSathesis} for RWDRE.
Note that, by considering time as an additional dimension, one-dimensional RWDRE can be seen as directed RWRE in two dimensions.

While one-dimensional RWRE is by now very well understood,
the state of the art in RWDRE is in comparison much more modest.
Most of the general results available require strong assumptions such as uniform and fast enough mixing
for the random environment, cf.\ e.g.\ \cite{AvdHoRe11}, \cite{dHodSaSi13}, \cite{ReVo11}.
An exception are quenched LDPs, cf.\ \cite{AvdHoRe10}, \cite{CDRRS12}, \cite{RaSeYi12}.
Otherwise, outside of the uniformly-mixing class, the literature is largely restricted to
particular choices of random environments, cf.\ e.g.\ \cite{HHSST14}, \cite{dHodSa13}, \cite{HS14}, \cite{MoVa15}, \cite{dSa14}.

In the present paper we consider the very basic question of whether the trichotomy between transience to the right,
transience to the left and recurrence,
typical for time-homogeneous Markov chains on $\Z$, also holds for one-dimensional, nearest-neighbour RWDRE.
We conclude that this is indeed the case under fairly general assumptions on the random environment.
An immediate but interesting consequence is that reflection-symmetric models satisfying our assumptions must be recurrent.
We will consider the continuous time setting, but the same arguments work, mutatis mutandis, in the discrete time case.
For comparison with other non-Markovian models where this problem was addressed, the reader is referred to \cite{zerner2001zero} and \cite{zerner2007zero} for the case of 2-dimensional RWRE, \cite{sabot2011reversed} for the case of random walks in Dirichlet environments, and to \cite{amir2013zero} for the case of 1-dimensional excited random walk.

The paper is organised as follows. In Section~\ref{s:results}
we define our model and state our assumptions and results.
Section~\ref{s:examplesanddiscussion} discusses our setup,
providing examples and connections to the literature;
the proof that a class of examples described therein
fits our setting is postponed to Section~\ref{s:proofofthmexamples}.
In Section~\ref{s:construction}, we present a graphical construction
that will be useful in Section~\ref{s:proofmainthm},
where our main theorem is proved.

%%%%%%%%%%%%%%%%%%%%%%%%%%%%% ASSUMPTIONS AND RESULTS %%%%%%%%%%%%%%%%%%%%%%%%%%%%%%%%%%%%%%%%%%%%%%%%%%%%%%%%%%%%%%%%%%%%%%%%%%
\section{Model, assumptions and results}
\label{s:results}

Let $\omega = (\omega^-_t, \omega^{+}_t)_{t \ge 0}$ be a stochastic process taking values on $([0,\infty)^{\Z})^2$, called
the \emph{dynamic random environment}.
We will assume that $\omega$ belongs to the space $\Omega$ of right-continuous paths from $[0,\infty)$ to $([0,\infty)^{\Z})^2$,
where the latter is endowed with the product topology.
Given a realisation of $\omega$, the RWDRE $X = (X_t)_{t \ge 0}$ is defined as the time-inhomogeneous
Markov jump process on $\Z$ whose laws $(P^\omega_x)_{x \in \Z}$ satisfy
\begin{align}\label{e:defX}
P^\omega_x \left(X_0=x \right) & = 1, \\
P^\omega_x\left(X_{t+s} = y \pm 1 \,\middle|\, X_t = y \right) & = s \, \omega^{\pm}_t(y) + o(s) \; \text{ as } s \downarrow 0.
\end{align}
The existence of such processes is standard (see e.g.\ \cite{EK86}, Chapter 4, Section 7).
We give here a particular construction in Section~\ref{s:construction} below.
Without extra assumptions the process $X$ may explode (i.e., make infinitely many jumps) in finite time;
we thus enlarge the state-space $\Z$ with a cemetery point $\Delta$ in the standard way in order to define $X$ after the explosion time $\tau_\Delta$,
i.e., $X_t := \Delta$ for all $t \ge \tau_\Delta$ (cf.\ \eqref{e:deftau_Delta}).

The law $P^\omega_x$ is called the \emph{quenched} law.
We denote by $\P_x$ the joint law of $X$ and $\omega$ (with $\P_x(X_0=x)=1$).
The corresponding expectations will be denoted respectively by $E^\omega_x$ and $\E_x$.
In the literature, the \emph{annealed} (or \emph{averaged}) law is often defined as the marginal law of $X$
under $\P_x$, but for convenience we will call $\P_x$ itself the annealed law.

Define the space-time translation operators $\theta^z_{s}:\Omega \to \Omega$, $z \in \Z, s \in \R_+$,
acting on $\omega$ as $(\theta^z_s \omega)_{x,t} := \omega(z+x,s+t)$.
We will write $\theta_s:=\theta^0_s$, $\theta^z:=\theta^z_0$.
Denoting by
\[\mathcal{F}_t:=\sigma(\omega, (X_u)_{0\le u \le t})\]
the natural filtration of $X$,
the Markov property for $X$ then reads
\begin{equation}\label{e:MPX}
E^{\omega}_x \left[ f \left((X_{t+s})_{s \ge 0} \right) \;\middle|\; \mathcal{F}_t \right]
= E_{X_t}^{\theta_t \omega} \left[ f(X)\right] \;\;\; P^\omega_x \text{-a.s.}
\end{equation}
for any bounded measurable $f$ and any $t \ge 0$.
Moreover, since the space-time process $(X_t,t)$ is Feller,
by the strong Markov property the time $t$ in \eqref{e:MPX} may be replaced by any a.s.\ finite $\mathcal{F}_t$-stopping time.
Also, we may and will assume that $X$ is right-continuous.

We will work under the following assumptions:

\textbf{(SE):}
\label{assumptionSE}
The process $\omega$ is stationary with respect to space-time translations, i.e.,
for each $z \in \Z, t \ge0$, $\theta^z_{t} \omega$ has the same distribution as $\omega$.
Furthermore, we assume that $\omega$ is ergodic with respect to the spatial translations $\theta^z$.

\textbf{(EL):} $\P_0$-a.s.,
\begin{equation}\label{e:EL}
\liminf_{t \to \tau_\Delta}{X_t} \;\textnormal{ and }\; \limsup_{t \to \tau_\Delta}{X_t} \in \{-\infty, + \infty\}.
\end{equation}

Assumption (SE) is standard; in fact, $\omega$ is usually taken ergodic also in time.
Assumption (EL) is an ellipticity condition;
note that it holds e.g.\ when $\omega$ is \emph{uniformly elliptic}, i.e.,
if there exists $\kappa \in (0,1)$ such that $\kappa \le \omega^{\pm}_t(x) \le \kappa^{-1}$.
Indeed, in this case the property of being visited infinitely often is either a.s.\ satisfied by all or by none of the points of $\Z$.
Note that (EL) implies
\begin{equation}\label{e:notstuck}
\inf \{ t > 0 \colon\, X_t \in [-n,n]^c \} < \tau_\Delta \;\;\;\; \P_0 \text{-a.s.\ for all } n \in \N.
\end{equation}
While (EL) may be hard to check in non-uniformly elliptic examples, \eqref{e:notstuck}
holds as soon as $\omega$ is stationary and ergodic with respect to time translations
and satisfies a non-degeneracy condition; see Proposition~\ref{prop:notstuck} below.

We can now state our main result.
\begin{theorem}
\label{thm:01law}
If assumptions (SE) and (EL) are satisfied, then
$\tau_\Delta = \infty$ $\P_0$-a.s.\ and one of the following three cases holds:
\begin{enumerate}
\item $\displaystyle \P_0\left(\lim_{t \to \infty} X_t = \infty\right)=1$ \label{transright};
\item $\displaystyle \P_0\left(\lim_{t \to \infty} X_t = -\infty \right)=1$ \label{transleft};
\item $\displaystyle \P_0 \left( \limsup_{t \to \infty} X_t = \infty = - \liminf_{t \to \infty} X_t \right) =1$. \label{rec}
\end{enumerate}
\end{theorem}

A zero-one law for directional transience is said to hold if the probabilities
in items \ref{transright} and \ref{transleft} of Theorem~\ref{thm:01law} are either $0$ or $1$.
This statement is equivalent to Theorem~\ref{thm:01law}
as the ellipticity assumption (EL) ensures that the event appearing in item \ref{rec} is almost surely equal
to the complement of the union of the events in \ref{transright}--\ref{transleft}.

As an immediate consequence of Theorem~\ref{thm:01law}, we obtain recurrence for any model satisfying (SE)--(EL)
that is symmetric with respect to reflection through the origin:
\begin{corollary}\label{cor:recurrencesymmetric}
Assume that (SE) and (EL) hold and that $(-X_t)_{t \ge 0}$ has under $\P_0$ the same distribution as $(X_t)_{t\ge 0}$.
Then item~\ref{rec} of Theorem~\ref{thm:01law} holds.
\end{corollary}
For an interesting example to which Corollary~\ref{cor:recurrencesymmetric} applies, consider the following.
Let $\omega^{+}_t(x) := \alpha \eta_t(x) + \beta (1-\eta_t(x))$, $\omega^-_t := \alpha (1-\eta_t(x)) + \beta \eta_t(x)$
where $0 < \beta < \alpha <\infty$ and
$(\eta_t)_{t \ge 0}$ is a simple symmetric exclusion process in $\Z$
started from a product Bernoulli measure $\nu_{\rho}$ with density $\rho \in (0,1)$.
Very little is known for this model in the literature (see e.g.\ \cite{AvdHoRe10}, \cite{HS14} and \cite{dSa14}),
in particular in the case $\rho=1/2$ where the expected asymptotic speed of $X$ is zero.
However, since it satisfies (SE)--(EL) and is reflection-symmetric for $\rho=1/2$,
Corollary~\ref{cor:recurrencesymmetric} implies that it is recurrent in this case.

\section{Examples and discussion}
\label{s:examplesanddiscussion}

In the literature, $\omega$ is often given as a functional of an interacting particle system,
i.e., of a Markov process $(\eta_t)_{t \ge 0}$ on $E^{\Z}$ where $E$ is a metric space, often assumed compact.
For example, in the setting of \cite{ReVo11}, the transition rates are given by
\begin{equation}\label{e:translratesFlo}
\omega^{\pm}_t(x) = \alpha^{\pm}(\theta^x \eta_t)
\end{equation}
where the functions $\alpha^{\pm}:E^{\Z} \to [0,\infty)$ satisfy some regularity properties.
The setting of \cite{AvdHoRe11} is a particular case where $E = \{0,1\}$.

Since directional transience follows from a law of large numbers with non-zero speed,
and recurrence from a functional central limit theorem if the speed is zero,
Theorem~\ref{thm:01law} brings no new information
in the cases where these results are known.
However, our result applies to many situations where such theorems
have not yet been proved, which is the case for several uniformly elliptic but non-uniformly mixing models,
e.g., when $\eta_t$ is a simple exclusion
process or a system of independent random walks outside the perturbative regimes considered in \cite{HHSST14}, \cite{HS14}.
By ``uniformly mixing'' here we mean that the conditional law of $\eta_t(0)$ given $\eta_0$
converges to a fixed law uniformly over all possible realizations of $\eta_0$;
cf.\ e.g.\ the cone-mixing condition  of \cite{AvdHoRe11} (Definition~1.1 therein),
or the coupling conditions of \cite{ReVo11} (Assumptions~1a--1b therein).

Let us now describe a large class of examples satisfying our assumptions that includes
many non-uniformly elliptic cases with slow and non-uniform mixing:
\begin{example}\label{ex:mainexample}
\textnormal{
Let $\eta_t(x)$, $x \in \Z, t \ge 0$, be i.i.d.\ in $x$ with each $\eta_t(x)$ distributed
as an irreducible, positive-recurrent Markov process on a countable state-space $E$,
started from its unique invariant probability measure $\pi$.
Let $\omega$ be defined by
$\omega^{\pm}_t(x) = \alpha^{\pm}(\eta_t(x))$
with $\alpha^{\pm}:E \to (0,\infty)$,
i.e., the jump rates are always positive (in which case the model is called \emph{elliptic})
and depend only on the state of $\eta_t$ at $x$.}
\end{example}

The models defined in Example~\ref{ex:mainexample} clearly satisfy (SE).
Moreover:

\begin{theorem}\label{thm:examplesatisfiesEL}
The models defined in Example~\ref{ex:mainexample} satisfy (EL).
\end{theorem}
\noindent
The proof of this theorem is given in Section~\ref{s:proofofthmexamples} below.
Note that, as already mentioned, it covers many models that
are slowly and non-uniformly mixing and thus do not fall into the categories
generally studied in the literature of RWDRE so far.

It is interesting to ask in which directions Theorem~\ref{thm:01law} could be generalised,
and how far our hypotheses could be weakened.
The analogous result in discrete time can be proved with a similar approach
via graphical representation (cf.\ Section~\ref{s:construction} below).
However, new ideas are needed for random walks in other graphs, e.g.\
$\Z^d$ with non-nearest neighbour jumps and/or $d > 1$, and regular trees.

\section{Graphical construction}
\label{s:construction}

We construct next a particular version of the process $X$ with convenient properties.
Denote by $\mathcal{M}_p$ the space of point measures on $\Z \times [0,\infty)$,
and let $N^+_\omega$, $N^-_\omega \in \mathcal{M}_p$ be two independent Poisson point processes
with intensity measures $\mu^{\pm}_\omega$ identified by
\begin{equation}\label{e:intPPPs}
\mu^{\pm}_\omega(A \times B) := \sum_{x \in A} \int_B \omega^{\pm}_s(x) ds, \;\;\; A \subset \Z, \, B \subset [0,\infty) \text{ measurable.}
\end{equation}
We denote by $\widehat{P}_\omega$ the joint law of $N^{+}_\omega$, $N^{-}_\omega$,
and by $\widehat{\P}$ the joint law of the latter and $\omega$.
Define the space-time translations $\theta^z_t$ of $N^{\pm}_\omega$ and functions thereof by
\begin{equation}\label{e:deftransN}
\begin{array}{lcll}
\theta^z_t N^{\pm}_\omega (C) & := & N^{\pm}_\omega(C+(z,t)), \;\;\; & C \subset \Z \times [0,\infty) \text{ measurable,}\\
\theta^z_t f(N^{\pm}_\omega) & := & f(\theta^z_t N^{\pm}_\omega), & f:\mathcal{M}_p \to \R,
\end{array}
\end{equation}
where
\[C+(z,t):= \bigcup_{(y,s) \in C}\{(y+z,s+t)\}.\]
We note that, under $\widehat{\P}$, $N^{\pm}_\omega$ inherits from $\omega$ the stationarity with respect to space-time translations
and the ergodicity with respect to spatial translations.

On each point of $N^{+}_\omega$, resp.\ $N^{-}_\omega$, we draw a unit-length arrow pointing to the right, resp.\ to the left.
Then we set, for $x \in \Z$, $X^{x}$ to be the path started at $x$ that proceeds by moving upwards in time and forcibly across any arrows
in a right-continuous way; the paths are defined only up to the explosion time.
See Figure~\ref{fig:graphrep}.

%%%%%%%%%%%%%%%%% FIGURE GRAPH REP %%%%%%%%%%%%%%%%%%%%%%%%%%%%
\begin{figure}[hbtp]
\vspace{1cm}
\begin{center}
\setlength{\unitlength}{0.3cm}
\begin{picture}(20,10)(0,0)
%horizontal lines
\put(-0.5,1){\line(23,0){23}} \put(-0.5,10){\line(23,0){23}}

%vertical lines
\put(2,1){\line(0,1){9}}
\put(5,1){\line(0,1){9}} \put(8,1){\line(0,1){9}}
\put(11,1){\line(0,1){9}} \put(14,1){\line(0,1){9}}
\put(17,1){\line(0,1){9}} \put(20,1){\line(0,1){9}}

%path x
{\linethickness{0.05cm}
\put(5,1){\line(0,1){1}} \put(5,2){\line(1,0){3}}
\put(8,2){\line(0,1){0.8}} \put(8,2.8){\line(-1,0){3}}
\put(5,2.8){\line(0,1){0.7}} \put(5,3.5){\line(-1,0){3}}
\put(2,3.5){\line(0,1){1.3}} \put(2,4.8){\line(1,0){3}}
\put(5,4.8){\line(0,1){0.8}} \put(5,5.6){\line(1,0){3}}
\put(8,5.6){\line(0,1){0.5}} \put(8,6.1){\line(1,0){3}}
\put(11,6.1){\line(0,1){0.9}} \put(11,7){\line(1,0){3}}
\put(14,7){\line(0,1){2.4}} \put(14,9.4){\line(-1,0){3}}
\put(11,9.4){\line(0,1){0.6}}
}

%arrows on path x
\put(5,2){\vector(1,0){3}}
\put(8,2.8){\vector(-1,0){3}}
\put(5,3.5){\vector(-1,0){3}}
\put(2,4.8){\vector(1,0){3}}
\put(5,5.6){\vector(1,0){3}}
\put(8,6.1){\vector(1,0){3}}
\put(11,7){\vector(1,0){3}}
\put(14,9.4){\vector(-1,0){3}}

%path y
{\linethickness{0.05cm}
\put(17,1){\line(0,1){0.6}} \put(17,1.6){\line(-1,0){3}}
\put(14,1.6){\line(0,1){0.8}} \put(14,2.4){\line(-1,0){3}}
\put(11,2.4){\line(0,1){1.3}} \put(11,3.7){\line(1,0){3}}
\put(14,3.7){\line(0,1){0,8}} \put(14,4.5){\line(1,0){3}}
\put(17,4.5){\line(0,1){0,5}} \put(17,5){\line(1,0){3}}
\put(20,5){\line(0,1){1}} \put(20,6){\line(-1,0){3}}
\put(17,6){\line(0,1){0,6}} \put(17,6.6){\line(1,0){3}}
\put(20,6.6){\line(0,1){1.1}} \put(20,7.7){\line(-1,0){3}}
\put(17,7.7){\line(0,1){1.3}} \put(17,9){\line(-1,0){3}}
}

%arrows on path y
\put(17,1.6){\vector(-1,0){3}}
\put(14,2.4){\vector(-1,0){3}}
\put(11,3.7){\vector(1,0){3}}
\put(14,4.5){\vector(1,0){3}}
\put(17,5){\vector(1,0){3}}
\put(20,6){\vector(-1,0){3}}
\put(17,6.6){\vector(1,0){3}}
\put(20,7.7){\vector(-1,0){3}}
\put(17,9){\vector(-1,0){3}}

%arrows off the paths
\put(2,3){\vector(-1,0){3}}
\put(2,7){\vector(-1,0){3}}
\put(2,2.4){\vector(1,0){3}}
\put(2,8.6){\vector(1,0){3}}

\put(5,6.4){\vector(-1,0){3}}
\put(5,4.2){\vector(1,0){3}}

\put(8,7.5){\vector(-1,0){3}}
\put(8,3.2){\vector(1,0){3}}
\put(8,9.1){\vector(1,0){3}}

\put(11,5.1){\vector(-1,0){3}}
\put(11,8.1){\vector(-1,0){3}}
\put(11,1.9){\vector(1,0){3}}
\put(11,5.7){\vector(1,0){3}}

\put(14,2.9){\vector(1,0){3}}
\put(14,6.3){\vector(1,0){3}}

\put(17,7.4){\vector(-1,0){3}}
\put(17,4.1){\vector(1,0){3}}
\put(17,9.5){\vector(1,0){3}}

\put(20,2.4){\vector(-1,0){3}}
\put(20,3.3){\vector(1,0){3}}
\put(20,8.5){\vector(1,0){3}}

%symbols
\put(-6.5,5.5){$\text{\small time}$}
\put(-3.5,4.5){\vector(0,1){3}}
\put(10.2,11){$X^x_t$}
\put(4.5,0){$x$}\put(16.5,0){$y$}
\put(-1.7,0.7){$0$}
\put(-1.7,9.7){$t$}
\put(23,0){$\mathbb{Z}$}
%\put(11,0){\circle*{.35}}
%\put(8,11){\circle*{.35}}
\end{picture}
\end{center}
\caption{\small Graphical construction. The arrows represent
events of $N^{\pm}_\omega$. The thick lines mark
the paths $X^x$ and $X^y$, which in this example coalesce at site $y-1$.} \label{fig:graphrep}
\end{figure}
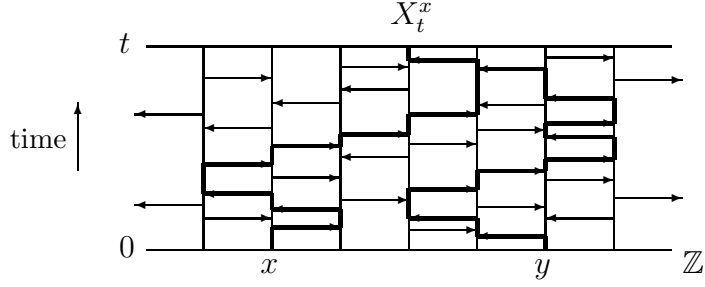
%%%%%%%%%%%%%%%%%%%%%%%%%%%%%%%%%%%%%%%%%%%%%%%%%%%%%%%%%%%%%%%%%%%%%%%%%%%%%%

Using the right-continuity of $\omega$, it is straightforward to check that this construction gives the correct law, i.e.,
$X^{x}$ has under $\widehat{P}_\omega$ the same law as $X$ under $P^\omega_x$.
In particular, this provides a coupling for copies of the random walk starting from all initial positions,
which will facilitate the proof of Theorem~\ref{thm:01law}.

With this construction, the explosion times $\tau^x_\Delta$, $x \in \Z$ can be defined as
\begin{equation}\label{e:deftau_Delta}
\tau^x_\Delta := \sup \{t > 0 \colon\, X^x \text{ crosses finitely many arrows up to time } t\},
\end{equation}
and we identify $X := X^0$, $\tau_\Delta := \tau^0_\Delta$.

We end this section with the following monotonicity property,
which is a consequence of the graphical construction and will be useful in the proof of Lemma~\ref{l:finmanyrexc} below.
\begin{lemma}\label{l:monot}
For any $y, z \in \Z$ such that $y \le z$, $\widehat{\P}$-a.s.,
\begin{equation}\label{e:monot}
X^y_t \le X^z_t \; \forall \; t \in [0, \tau^y_\Delta \wedge \tau^z_\Delta).
\end{equation}
\end{lemma}
\begin{proof}
Since the paths start ordered, move by nearest-neighbour jumps, and a.s.\ cannot jump simultaneously before they meet,
either $X^y_t < X^z_t$ for all relevant $t$ or there exists a first $s\ge 0$ such that $X^y_s = X^z_s$, in which case by construction $X^y_u = X^z_u$ for all $u \ge s$.
\end{proof}

%%%%%%%%%%%%%%%%%%%%%%%%%%%%%%%% PROOF MAIN THEOREM %%%%%%%%%%%%%%%%%%%%%%%%%%%%%%%%

\section{Proof of Theorem~\ref{thm:01law}}
\label{s:proofmainthm}

For $A \subset \Z$, denote by
\begin{equation}\label{e:defhittimes}
H_A := \inf \{t > 0 \colon\, X_t \in A \}
\end{equation}
the hitting time of $A$.
Let $A^c := \Z \setminus A$ and note that, if $A$ is finite, then $H_{A^c}$ is a.s.\ finite by \eqref{e:notstuck}.
For a random time $S \in [0,\infty]$, define
\begin{equation}\label{e:defthetaS}
\Theta_S H_A := \left\{ \begin{array}{ll} \inf\{ t > 0 \colon\, X_{S+t} \in A \} & \text{ if } S < \infty,\\
\infty & \text{ otherwise.} \end{array}\right.
\end{equation}
Note that $\Theta_S H_A = \theta^{X_S}_S H_{A-X_S}$ when $S<\infty$,
where $A-x:=\{z -x \colon\, z \in A\}$.
Define now the $k$-th return time $T^{(k)}_A$ to $A$ as follows.
Set $T^{(0)}_A:= 0$ and, recursively for $k \ge 0$,
\begin{equation}\label{e:defhittimeslater}
T^{(k+1)}_A := T^{(k)}_A +  \Theta_{T^{(k)}_A} \left(H_{A^c} + \Theta_{H_{A^c}} H_A \right).
%T^{(k+1)}_A := T^{(k)}_A +  \Theta_{T^{(k)}_A} T_{A^c} + \Theta_{T^{(k)}_A + \Theta_{T^{(k)}_A} T_{A^c}} T_A.
%\left\{ \begin{array}{ll}
%\inf \{t>T^{(k-1)}_z \colon X_t = z\} & \text{ if } T^{(k-1)}_z < \infty,\\
%\infty & \text{ otherwise,}
%\end{array}\right.
\end{equation}
Note that $T^{(1)}_A = H_A$ if $X_0 \notin A$. When $A = \{z\}$, we write $H_z$ and $T^{(k)}_z$.

Our proof of Theorem~\ref{thm:01law} is based on three lemmas which we state next;
their proofs are given respectively in Sections~\ref{ss:thetranslationinvariancelemma} and~\ref{ss:the1dimlemma} below.
The first of them implies that, if the random walk visits $-1$ (resp.\ $1$) a.s.,
then all its excursions from $0$ to the right (resp.\ to the left) will be a.s.\ finite.
\begin{lemma}\label{l:rightexcarefinite}
Assume that (SE) holds, and let $x \in \{-1, 1\}$.
If $\P_0(H_{x} < \infty) = 1$, then
\begin{equation}\label{e:rightexcarefinite}
T^{(k)}_0 < \infty \;\; \Rightarrow \;\; T^{(k)}_0 + \Theta_{T^{(k)}_0} H_{x} < \infty \;\;\;\; \P_0 \text{-a.s.\ for all } k \ge 1.
\end{equation}
\end{lemma}

The second lemma excludes the possibility of explosions in our setting.
\begin{lemma}\label{l:noexplosion}
Assume that (SE) and \eqref{e:notstuck} hold. Then
\begin{equation}\label{e:noexplosion}
\P_0 \left( \tau_\Delta = \infty \right)=1.
\end{equation}
\end{lemma}

The third lemma shows that, if there is a positive probability for the random walk to never touch $-1$ (resp.\ $1$),
then its range is bounded from below (resp.\ above).
\begin{lemma}\label{l:finmanyrexc}
Assume that (SE) and \eqref{e:notstuck} hold.
Then $\P_0(H_{-1} = \infty) >0$ implies $\P_0 \left( \exists \, z < 0 \colon\, H_z = \infty \right)=1$,
and $\P_0(H_{1} = \infty) >0$ implies $\P_0\left(\exists \, z > 0 \colon\, H_z = \infty \right)=1$.
\end{lemma}

Note that Lemmas~\ref{l:rightexcarefinite}--\ref{l:finmanyrexc} do not use assumption (EL) directly
but only its consequence \eqref{e:notstuck}.
Moreover, Lemma~\ref{l:rightexcarefinite} only uses stationarity in time and the strong Markov property;
the graphical construction of Section~\ref{s:construction} is only used in the proof of Lemmas~\ref{l:noexplosion} and~\ref{l:finmanyrexc}.

We can now finish the proof of Theorem~\ref{thm:01law}.
\begin{proof}[Proof of Theorem~\ref{thm:01law}]
Assumption (EL) and Lemmas~\ref{l:noexplosion}--\ref{l:finmanyrexc} together imply that
\begin{equation}\label{e:prfmainthm1}
\P_0(H_{-1}=\infty)>0 \;\; \Rightarrow \;\; \P_0 \left( \lim_{t\to \infty} X_t = \infty \right) = 1
\end{equation}
since, if the left-hand side of \eqref{e:prfmainthm1} holds, then $\liminf_{t \to \infty} X_t > - \infty$ a.s.\ and hence it must be equal to $\infty$
by (EL). Analogously,
\begin{equation}\label{e:prfmainthm2}
\P_0(H_{1}=\infty)>0 \;\; \Rightarrow \;\; \P_0 \left( \lim_{t\to \infty} X_t = -\infty \right) = 1.
\end{equation}
To conclude, we claim that
\begin{equation}\label{prmainthm3}
\P_0\left( H_1 \vee H_{-1} < \infty \right) = 1 \;\; \Rightarrow \;\; -\infty = \liminf_{t \to \infty} X_t < \limsup_{t \to \infty} X_t = \infty.
\end{equation}
Indeed, note that, by Lemmas~\ref{l:rightexcarefinite}--\ref{l:noexplosion}, $\P_0\left(H_{-1} < \infty \right) = 1$ implies that $\liminf_{t \to \infty} X_t \le -1$ a.s., which together with (EL) gives $\liminf_{t \to \infty} X_t = -\infty$.
The last equality is obtained analogously.
\end{proof}

\subsection{Proof of Lemma~\ref{l:rightexcarefinite}}
\label{ss:thetranslationinvariancelemma}

\begin{proof}
To start, we claim that, $\P_0$-a.s.,
\begin{equation}\label{e:samepropforallt}
P^{\theta_t \omega}_0 (H_{x} = \infty) = 0 \; \text{ simultaneously for all } t \ge 0.
\end{equation}
Indeed, for each fixed $t \ge 0$, $P^{\theta_t \omega}_0 (H_{x} = \infty) = 0$ a.s.\ since, by stationary,
$\P_0(\cdot) = \E_0[P^\omega_0(\cdot)] = \E_0[P^{\theta_t \omega}_0(\cdot)]$.
Hence \eqref{e:samepropforallt} holds with $t$ restricted to the set of rational numbers,
and to extend it to all $t\ge0$ we only need to show that
the function $t \mapsto P^{\theta_t \omega}_0(H_x = \infty)$ is right-continuous.
To this end, note that, since $\omega$ is right-continuous,
\begin{align}\label{e:prlemmaMKV1}
P^{\theta_t \omega}_0 \left(\exists \, u \in [0,s] \colon X_u \neq 0 \right)
& = 1-e^{-\int_t^{t+s} \left\{ \omega^+_u(0)+\omega^{-}_u(0) \right\} du} \nonumber\\
& \le \int_t^{t+s} \left\{ \omega^+_u(0)+\omega^{-}_u(0) \right\} du \nonumber\\
& \le 2 s \left\{ \omega^+_t(0) + \omega^-_t(0) \right\}
\end{align}
for all $s>0$ small enough (depending on $\omega$ and $t$).
Denoting by $O_\omega(s)$ a function whose modulus is bounded by $s \, C_\omega$
where $C_\omega \in (0,\infty)$ may depend on $\omega$, we obtain
\begin{align}\label{e:prlemmaMKV2}
P^{\theta_t \omega}_0 \left(H_x = \infty \right)
& = P^{\theta_t \omega}_0 \left(\Theta_s H_x = \infty, X_u = 0 \,\forall\, u \in [0,s] \right) + O_{\omega}(s) \nonumber\\
& = E^{\theta_t \omega}_0 \left[ \mathbbm{1}_{\{X_u = 0 \,\forall\, u \in [0,s] \}}P^{\theta_{t+s} \omega}_0 \left( H_x = \infty \right)\right] + O_{\omega}(s) \nonumber\\
& = P^{\theta_{t+s}\omega}_0 \left( H_x = \infty \right) + O_{\omega}(s),
\end{align}
where for the second line we use the Markov property and for the last one we again use \eqref{e:prlemmaMKV1}.
From this follows the desired right-continuity and consequently also \eqref{e:samepropforallt}.
By the strong Markov property (cf.\ the paragraph below \eqref{e:MPX}) and $\Theta_{T^{(k)}_0}H_x = \theta_{T^{(k)}_0}H_x$,
\begin{align}\label{e:prlemmaMKV3}
\P_0 \left( T^{(k)}_0 < \infty, \Theta_{T^{(k)}_0}H_x = \infty\right)
& = \E_0 \left[\int_0^\infty P^{\theta_t \omega}_0 \left( H_x = \infty \right) P^\omega_0 \left( T^{(k)}_0 \in d t \right) \right] = 0
\end{align}
by \eqref{e:samepropforallt}.
\end{proof}

\subsection{Proof of Lemmas~\ref{l:noexplosion}--\ref{l:finmanyrexc}}
\label{ss:the1dimlemma}

We start by showing that explosions are not possible under (SE) and \eqref{e:notstuck}.
\begin{proof}[Proof of Lemma~\ref{l:noexplosion}]
\text{}

\noindent
It is enough to show that, for any $a,b \in [0,\infty)$ with $b-a>0$ small enough,
\begin{equation}\label{e:prnoexpl1}
\P_0 \left( \tau_\Delta \in (a,b] \right) = 0.
\end{equation}
To that end, define the events
\begin{equation}\label{e:prnoexpl2}
A_x^{a,b} := \left\{ \text{there are no arrows in } \{x\} \times [a,b]\right\} = \left\{N^{\pm}_\omega\left(\{x\} \times [a,b]\right) = 0 \right\}
\end{equation}
and let $\varepsilon > 0$ be so small that, if $b-a \le \varepsilon$, then
\begin{equation}\label{e:prnoexpl3}
\widehat{\P} \left( A_0^{0,b-a}\right) \ge
\widehat{\P} \left( A_0^{0,\varepsilon} \right) = \widehat{\E} \left[\exp\left\{- \int_0^\varepsilon \left[ \omega^+_s(0)+\omega^-_s(0)\right]ds \right\} \right]> 0,
\end{equation}
which exists by the right-continuity of $\omega$ and the dominated convergence theorem.
Noting that $A_x^{a,b} = \theta^x A_0^{a,b}$,
we obtain from Birkhoff's ergodic theorem that, $\widehat{\P}$-a.s.,
\begin{equation}\label{e:prnoexpl4}
\lim_{N \to \infty} \frac{1}{N} \sum_{x =0}^{N-1} \mathbbm{1}_{A_{x}^{a,b}} = \widehat{\P} \left(A_0^{a,b} \right) = \widehat{\P} \left(A_0^{0,b-a} \right) >0
\end{equation}
by stationarity under time translations, and analogously for $x \le 0$.
In particular,
\begin{equation}\label{e:prnoexpl5}
\widehat{\P} \left( \forall\, z \in \Z, \;\exists\, x < z < y \text{ such that } A_x^{a,b} \text{ and } A_y^{a,b} \text{ occur}\right) =1.
\end{equation}
Note now that, by \eqref{e:notstuck},
if $\tau_\Delta \in (a,b]$ then for all $n \in \N$ the random walk exits the interval $[-n,n]$ before time $b$.
Therefore
\begin{equation}\label{e:prnoexpl6}
\P_0 \left( \tau_\Delta \in (a,b] \right) \le \widehat{\P} \left( \forall\, n \in \N, a+\Theta_a H_{[-n,n]^c} < b\right)
\end{equation}
where $H_{[-n,n]^c}$ is the hitting time of $\Z \setminus [-n,n]$ by $X^0$.
On the other hand, by the graphical construction,
if both $A_x^{a,b}$ and $A_y^{a,b}$ occur for some $x < X^0_a <  y$, then $a+\Theta_a H_{[-n,n]^c} \ge b$ with
$n = |x| \vee |y|$.
Hence \eqref{e:prnoexpl6} is at most
\begin{equation}\label{e: prnoexpl7}
\widehat{\P} \left( \forall\, x,y \in \Z \text{ such that } x< X^0_a < y, \text{ either } A_x^{a,b} \text{ or } A_y^{a,b} \text{ does not occur} \right)=0
\end{equation}
by \eqref{e:prnoexpl5}, proving \eqref{e:prnoexpl1}.
\end{proof}

We now prove Lemma~\ref{l:finmanyrexc}.
\begin{proof}[Proof of Lemma~\ref{l:finmanyrexc}]
We assume that $\P_0 (H_{-1} = \infty) >0$;
the case $\P_0 (H_1 = \infty) >0$ is proved analogously.
For $x \in \Z$, let
\begin{equation}\label{e:prlemma1D1}
A_x := \left\{ X^x_t \ge x \, \forall \, t \ge 0 \right\}.
\end{equation}
Since $A_x = \theta^x A_0$,
Birkhoff's ergodic theorem implies that, $\widehat{\P}$-a.s.,
\begin{equation}\label{e:prlemma1D2}
\lim_{N \to \infty} \frac{1}{N} \sum_{z =0}^{-N+1} \mathbbm{1}_{A_z} = \widehat{\P} \left(A_0 \right) = \P_0 \left( H_{-1} = \infty \right) >0,
\end{equation}
and in particular $\widehat{\P}(A_z \text{ occurs for some } z \le 0)=1$.
Noting that, by Lemmas~\ref{l:monot} and~\ref{l:noexplosion},
if $z \le 0$ and $A_z$ occurs then $X^0_t > z-1$ for all $t \ge 0$,
we obtain
\begin{align}\label{e:prlemma1D3}
\P_0 \left( \exists\, z < 0 \colon\, H_z = \infty \right)
& = \widehat{\P} \left( \exists \, z < 0 \colon\, X^0_t \neq z \,\forall\, t \ge 0 \right) \nonumber\\
& \ge \widehat{\P} \left( \exists \, z \le 0 \colon\, A_z \text{ occurs } \right) = 1,
\end{align}
finishing the proof.
\end{proof}

\section{Proof of Theorem~\ref{thm:examplesatisfiesEL}}
\label{s:proofofthmexamples}

We first show that \eqref{e:notstuck} holds for a very large class of models,
including our examples.
\begin{proposition}\label{prop:notstuck}
Assume that $\omega$ is stationary and ergodic with respect to the time translation $\theta_1$
and that, for a choice of $*,\star \in \{-,+\}$ and every $n \in \N$,
\begin{align}
& \P_0 \left( \int_0^{n} \omega^*_s(0) \, ds < \infty \right) = 1 \;\; \textnormal{ and }\label{e:nondeg1}\\
& \P_0 \left(\omega^{\star}_0(x) > 0 \;\forall\, x \in [-n,n] \right) >0. \label{e:nondeg2}
\end{align}
Then \eqref{e:notstuck} holds.
\end{proposition}
\begin{proof}
Fix $n \in \N$.
By right-continuity and invariance under time translations,
there exist $\delta, \varepsilon \in (0,1)$ such that the events
\begin{equation}\label{e:prnotstuck1}
\mathcal{A}_k := \left\{\omega^\star_{k+s}(x) \ge \delta \;\forall\, s \in [0,\varepsilon], x \in [-n,n] \right\}, \;\; k \in \N
\end{equation}
have equal and positive probability.
Then the event
\begin{equation}\label{e:prnotstuck2}
\mathcal{A} := \limsup_{k \to \infty} \mathcal{A}_k = \bigcap_{k \ge 1} \bigcup_{l \ge k} \mathcal{A}_k = \left\{ \mathcal{A}_k \text{ occurs for infinitely many } k \in \N\right\}
\end{equation}
also has positive probability by the Poincar\'e recurrence theorem (cf.\ Theorem 1.4 in \cite{Wa82});
moreover, since $\mathcal{A}$ is invariant under $\theta_1$, it occurs almost surely by ergodicity.
Let
\begin{equation}\label{e:prnotstuck3}
\begin{array}{lcl}
V_1 &:= & \inf\left\{l \in \N \colon\, \mathcal{A}_l \text{ occurs}\right\},\\
V_{k+1} & := & \inf\left\{l > V_k \colon\, \mathcal{A}_l \text{ occurs}\right\}, \; k \ge 1.
\end{array}
\end{equation}
Denote by $H_{[-n,n]^c}$ the hitting time of $\Z \setminus [-n,n]$.
By \eqref{e:nondeg1}, $N^{*}_\omega([-n,n] \times [0,T])<\infty$ a.s.\ for all $T \ge 0$,
and thus $H_{[-n,n]^c} = \infty$ implies $\tau_{\Delta} = \infty$ almost surely.
Hence
\begin{align}\label{e:prnotstuck4}
\P_0 \left( H_{[-n,n]^c} = \infty \right)
& \le \P_0 \left( X_{V_k} \in [-n,n], \Theta_{V_k} H_{[-n,n]^c} > \varepsilon \; \forall \, k \ge 1\right) \nonumber\\
& = \lim_{L \to \infty} \P_0 \left( X_{V_k} \in [-n,n], \Theta_{V_k} H_{[-n,n]^c} > \varepsilon \; \forall \, 1 \le k \le L \right).
\end{align}
Note now that, if $X_{V_k} \in [-n,n]$, then between times $V_k$ and $V_k + \varepsilon \wedge \Theta_{V_k} H_{[-n,n]^c}$
the RWDRE has a rate at least $\delta$ to jump in direction $\star$. Therefore,
\begin{equation}\label{e:prnotstuck5}
X_{V_k} \in [-n,n] \; \Rightarrow \; P_{X_{V_k}}^{\theta_{V_k}\omega} \left(H_{[-n,n]^c} \le \varepsilon  \right) \ge \vartheta_n
\end{equation}
for some deterministic $\vartheta_n \in(0,1)$ independent of $k$.
By the Markov property,
\begin{align}\label{e:prnotstuck6}
& \P_0 \left( X_{V_k} \in [-n,n], \Theta_{V_k} H_{[-n,n]^c} > \varepsilon \; \forall \, 1 \le k \le L+1 \right) \nonumber\\
= \; &  \E_0 \left[\mathbbm{1}_{\{ X_{V_k} \in [-n,n], \Theta_{V_k} H_{[-n,n]^c} > \varepsilon \; \forall \, 1 \le k \le L \}} P^{\theta_{V_L} \omega}_{X_{V_L}} \left( H_{[-n,n]^c} > \varepsilon \right)\right] \nonumber\\
 \le \; &(1-\vartheta_n) \P_0 \left( X_{V_k} \in [-n,n], \Theta_{V_k} H_{[-n,n]} > \varepsilon \; \forall \, 1 \le k \le L \right)
\end{align}
and we conclude using induction that \eqref{e:prnotstuck4} is equal to zero, proving \eqref{e:notstuck}.
\end{proof}

As a consequence, no explosion can occur in our examples.
\begin{corollary}\label{cor:noexpexamples}
In all models defined in Example~\ref{ex:mainexample},
 $\tau_\Delta = \infty$ almost surely.
\end{corollary}
\begin{proof}
The models described are mixing in time, and thus satisfy the hypotheses of Proposition~\ref{prop:notstuck}.
Since they also satisfy (SE), the corollary follows by Lemma~\ref{l:noexplosion}.
\end{proof}

We can now finish the proof of Theorem~\ref{thm:examplesatisfiesEL}.
\begin{proof}[Proof of Theorem~\ref{thm:examplesatisfiesEL}]
We will prove that, $\P_0$-a.s.,
\begin{equation}\label{e:prexamples1}
\forall \; x \in \Z, \;\; T^{(k)}_x < \infty \;\; \forall \; k \ge 1  \; \Rightarrow \; T^{(k)}_{x-1} < \infty \; \; \forall \; k \ge 1.
\end{equation}
The analogous result for $x+1$ in place of $x-1$ then follows by reflection.
This implies (EL) since, by Proposition~\ref{prop:notstuck}, \eqref{e:notstuck} holds.
We claim that it suffices to show that, a.s.,
\begin{equation}\label{e:prexamples2}
\forall \; x \in \Z, \;\; T^{(k)}_x < \infty \;\; \forall \; k \ge 1  \; \Rightarrow \; T^{(1)}_{x-1} < \infty.
\end{equation}
Indeed, fix $j \in \N$.
Suppose by induction that \eqref{e:prexamples2} holds with $(1)$ substituted by $(j)$.
Then write, using the strong Markov property,
\begin{align}\label{e:prexamples3}
& \P_0 \left( T^{(k)}_x < \infty \; \forall \; k \ge 1, T^{(j+1)}_{x-1} = \infty \right) \nonumber\\
= \; & \P_0 \left( \Theta_{T^{(j)}_{x-1}} T^{(k)}_x < \infty \; \forall \; k \ge 1, T^{(j)}_{x-1} < \infty, \Theta_{T^{(j)}_{x-1}} T^{(1)}_{x-1} = \infty \right) \nonumber\\
= \; & \E_0 \left[ \int_0^{\infty} P^{\theta_t \omega}_{x-1} \left( T^{(k)}_x < \infty \; \forall \; k \ge 1,  T^{(1)}_{x-1} = \infty \right) P^\omega_0 \left(T^{(j)}_{x-1} \in d t \right)\right].
\end{align}
With an argument identical to the one used to prove \eqref{e:samepropforallt},
we can show that the integrand in \eqref{e:prexamples3} is a.s.\ identically equal to $0$, proving \eqref{e:prexamples1}.

Fix $\mathcal{O} \in E$ and choose a finite set $E^* \subset E$ such that $\mathcal{O} \in E^*$ and
\begin{equation}\label{e:prexamples5}
\inf_{t \ge 0} \P_0 \left( \eta_t(0) \in E^*   \,\middle|\, \eta_0(0) = \mathcal{O} \right) > \tfrac12.
\end{equation}
This is possible since $\eta_t(0)$ converges in distribution to $\pi$ (cf.\ Theorem 2.66 in \cite{Li10}).
Considering the maximal jump rate in $E^*$,
we further obtain $\varepsilon>0$ such that
\begin{equation}\label{e:prexamples6}
\inf_{t \ge 0} \P_0 \left( \eta_s(0) \in E^*  \,\forall \, s \in [t,t+\varepsilon]  \,\middle|\, \eta_0(0) = \mathcal{O} \right) > \tfrac12.
\end{equation}

Fix now a site $x \in \Z$ and define
\begin{equation}\label{e:prexamples7}
\begin{array}{lcl}
U_1 & := & \inf \{t > 0 \colon\, \eta_t(x) = \mathcal{O}\},\\
V_1 & := & \inf \{t > U_1 \colon\, X_t = x \} =  U_1 + \Theta_{U_1}T^{(1)}_x,
\end{array}
\end{equation}
and, recursively for $k \ge 2$,
\begin{equation}\label{e:prexamples8}
\begin{array}{lcl}
U_{k} & := & \left\{ \begin{array}{ll} \inf \{t > V_{k-1}+1 \colon\, \eta_t(x) = \mathcal{O}\} & \qquad \quad \; \text{ if } V_{k-1} < \infty, \\
                                      \infty & \qquad \quad \; \text{ otherwise,} \end{array}\right. \\
V_{k} & := & \left\{ \begin{array}{ll} \inf \{t > U_k \colon\, X_t = x \} =  U_k + \Theta_{U_k}T^{(1)}_x & \text{ if } U_{k} < \infty, \\
                                      \infty & \text{ otherwise.}
\end{array}\right.
\end{array}
\end{equation}
These are all well-defined by Corollary~\ref{cor:noexpexamples}.
Note that $T^{(k)}_x< \infty$ for all $k \ge 1$ if and only if $V_k < \infty$ for all $k \ge 1$.
Therefore, it is enough to show that
\begin{equation}\label{e:prexamples9}
\P_0 \left( V_k < \infty , \Theta_{V_i} T^{(1)}_{x-1} \ge 1 \;\forall\, 1 \le i \le k \right) \le \rho^{k-1}
\end{equation}
for some $\rho \in (0,1)$.
To this end, note first that
\begin{equation}\label{e:prexamples10}
\eta_{t+s}(x) \in E^* \;\forall\,  s \in [0,\varepsilon] \; \Rightarrow \; P_x^{\theta_t \omega} \left( T^{(1)}_{x-1} \le \varepsilon \right) \ge \delta
\end{equation}
for some deterministic $\delta> 0$, since the left-hand side of \eqref{e:prexamples10} implies that $\omega^{\pm}_{t+s}(x)$ is bounded away from zero and infinity uniformly in $s \in [0,\varepsilon]$.
Therefore, for any initial configuration $\bar{\eta} \in E^{\Z}$ and any $z \in \Z$,
\begin{align}\label{e:prexamples11}
& \P_z \left( V_1 < \infty, \Theta_{V_1}T^{(1)}_{x-1} < 1 \,\middle|\, \eta_0 = \bar{\eta} \right) \nonumber\\
\ge \; &  \P_z \left(  V_1 < \infty, \eta_{V_1+s}(x) \in E^* \;\forall\, s \in [0,\varepsilon], \Theta_{V_1} T^{(1)}_{x-1} \le \varepsilon \,\middle|\, \eta_0 = \bar{\eta} \right) \nonumber\\
= \; & \E_z \left[  \mathbbm{1} \{ V_1 < \infty, \eta_{V_1+s}(x) \in E^* \;\forall\, s \in [0,\varepsilon] \} P_x^{\theta_{V_1} \omega} \left( T^{(1)}_{x-1} \le \varepsilon \right) \,\middle|\, \eta_0 = \bar{\eta} \right]\nonumber\\
\ge \; & \delta \, \P_z \left(  V_1 < \infty, \eta_{V_1+s}(x) \in E^* \;\forall\, s \in [0,\varepsilon] \,\middle|\, \eta_0 = \bar{\eta}\right) \nonumber\\
= \; & \delta \, \P_z \left( U_1 \le V_1 < \infty, \eta_{U_1+(V_1 -U_1) + s}(x) \in E^* \;\forall\, s \in [0,\varepsilon] \,\middle|\, \eta_0 = \bar{\eta}\right).
\end{align}
Note now that $U_1$, $V_1$ are measurable in
\[
\sigma \left( N^{\pm}_\omega(C) \colon\, C \subset \{x\} \times [0,U_1] \cup (\Z\setminus\{x\}) \times [0,\infty) \right),
\]
while $(\eta_{U_1+s}(x))_{s \ge 0}$ is independent of the latter sigma-algebra with distribution equal to that of $(\eta_s(x))_{s \ge 0}$ with $\eta_0(x) = \mathcal{O}$.
Therefore, \eqref{e:prexamples11} equals
\begin{align}\label{e:prexamples20}
& \delta \int \hspace{-7pt} \int_{0<u\le v < \infty} \hspace{-9pt} \P_z \left( \cap_{s \in [0,\varepsilon]}\{ \eta_{v-u + s}(x) \in E^*\} \,\middle|\, \eta_0(x) = \mathcal{O}\right) \P_z \left(V_1 \in dv, U_1 \in du \,\middle|\, \eta_0 = \bar{\eta} \right) \nonumber\\
& \ge \frac{\delta}{2} \P_z \left( V_1 < \infty \,\middle|\, \eta_0 = \bar{\eta} \right)
\end{align}
by \eqref{e:prexamples6}, implying that, for any $\bar{\eta} \in E^{\Z}$ and any $z \in \Z$,
\begin{equation}\label{e:prexamples13}
\P_z \left( V_1 < \infty, \Theta_{V_1} T^{(1)}_{x-1} \ge 1 \,\middle|\, \eta_0 = \bar{\eta} \right) \le \rho := 1 - \frac{\delta}{2} < 1.
\end{equation}
To conclude, use the strong Markov property of $(X_t, \eta_t)$ at time $V_k+1$ to write
\begin{align}\label{e:prexamples14}
& \P_0 \left( V_{k+1} < \infty, \Theta_{V_i} T^{(1)}_{x-1} \ge 1 \;\forall\, 1 \le i \le k+1\right) \nonumber\\
= \; & \E_0 \left[ \mathbbm{1}_{\{V_{k} < \infty, \Theta_{V_i} T^{(1)}_{x-1} \ge 1 \;\forall\, 1 \le i \le k\}} \P_{X_{V_k+1}} \left( V_1 < \infty, \Theta_{V_1} T^{(1)}_{x-1} \ge 1 \,\middle|\, \eta_0 = \bar{\eta}\right)_{\bar{\eta} = \eta_{V_k+1}}\right] \nonumber\\
\le \; & \rho \, \P_0 \left( V_{k} < \infty, \Theta_{V_i} T^{(1)}_{x-1} \ge 1 \;\forall\, 1 \le i \le k\right)
\end{align}
by \eqref{e:prexamples13}, and so \eqref{e:prexamples9} follows by induction.
\end{proof}

%%%%%%%%%%%%%%%%%%%%%%% ACKKNOWLEDGEMENTS
\bigskip
\noindent
\textbf{Acknowledgments.}
The authors would like to thank Luca Avena for the suggestion to add Corollary~2.2,
and the anonymous referee for useful comments. 
The work of TO was supported by the Labex Milyon (ANR-10-LABX-0070) of Universit\'e de Lyon, within the program "Investissements d'Avenir" (ANR-11-IDEX-0007) operated by the French National Research Agency (ANR).
RSdS was supported by the German DFG project KO 2205/13-1 ``Random mass flow through random potential''.

%%%%%%%%%%%%%%%%%%%%%%%%%%%%%%%%%%%% REFERENCES %%%%%%%%%%%%%%%%%%%%%%%%%%%%%%%%%%%%%%%%%%%%%%%%%%%%%%%%%%%%%%%%%%%

\end{document}